\documentclass[a4paper,10pt]{amsart}

\usepackage[utf8]{inputenc}
\newtheorem{thm}{Theorem}[section]
\newtheorem{prop}[thm]{Proposition}
\newtheorem{lemma}[thm]{Lemma}
\newtheorem{cor}[thm]{Corollary}
\theoremstyle{definition}
\newtheorem{remark}[thm]{Remark}
\newtheorem{ex}[thm]{Example}

\numberwithin{equation}{section}

\title{Leibniz seminorms in probability spaces}
\author{\'Ad\'am Besenyei and Zolt\'an L\'eka}

\address{Department of Applied Analysis, E\"otv\"os Lor\'and University,
H-1117 Budapest, P\'azm\'any P. s\'et\'any 1/C, Hungary}
\email{badam@cs.elte.hu}

\address{Alfr\'ed R\'enyi Institute of Mathematics \\ 1053 Budapest \\ Re\'altanoda u. 13-15}
\email{leka.zoltan@renyi.mta.hu}

\thanks{This study was partially supported by the Hungarian NSRF (OTKA) grant
no. K104206 }

\subjclass[2000]{Primary 46L53, 60E15 ; Secondary 26A51, 60B99.}
\keywords{standard deviation, Leibniz seminorm, central moments, $C^*$-algebra}

\date{}

\begin{document}
 
 \begin{abstract}
   In this paper we study the (strong) Leibniz property of centered moments of bounded random variables. 
   We shall answer a question raised by M. Rieffel on the non-commutative standard deviation.  
 \end{abstract}

 \maketitle
 
\section{Introduction}


  We say that a seminorm $L$ on a unital normed algebra
$(\mathcal{A}, \|\cdot \|)$ is strongly Leibniz if (i) $L(1_\mathcal{A}) = 0,$ (ii) the Leibniz property
 $$ L(ab) \leq \|a\|L(b) + \|b\| L(a) $$ holds for every $a, b \in \mathcal{A}$ and, furthermore, 
 (iii) for every invertible $a,$ 
  $$ L(a^{-1}) \leq \|a^{-1}\|^2 L(a)$$ follows. Primary sources of strongly Leibniz seminorms are normed
  first-order differential calculi, see \cite{R2}. It is said that the 
  couple $(\Omega, \delta)$ is a normed first-order differential calculus over $\mathcal{A}$ if $\Omega$
  is a normed bimodule over $\mathcal{A}$ and $\delta$ is a derivation from $\mathcal{A}$ to $\Omega.$
  Now let us assume that $\Omega$ is acting boundedly over $\mathcal{A};$ that is, the inequalities
  $$\|a\omega\| \leq \|\omega\|_\Omega \|a\| \quad \mbox{ and } \quad  \|\omega a\| \leq \|\omega\|_\Omega \|a\| $$
  hold for every $\omega \in \Omega$ and for every $a \in \mathcal{A}.$ From the derivation rule
  $$ \delta(ab) = \delta(a)b+a\delta(b),$$ the Leibniz property of the seminorm $L(a)=\|\delta(a)\|_\Omega$ simply follows.
  Furthermore, we clearly have that 
   $$ \delta(a^{-1}) = -a^{-1}\delta(a)a^{-1},$$ whenever $a$ is invertible, hence (iii) follows as well. For instance, if we choose a (real or complex) Banach space $X$ and  $\mathcal{B}(X)$ denotes
 the normed algebra of its bounded linear operators, practically, we can easily get a first-order differential calculus.
 Actually, with the choice of $\Omega = \mathcal{B}(X),$ which acts naturally over $\mathcal{B}(X)$ via
 the left and right multiplications, the commutator $\delta(A) = [D,A] = DA - AD$ for some fixed $D \in \mathcal{B}(X)$  defines the required calculus.
 
   Consider a unital $C^*$-algebra $\mathcal{A}$ and denote $\mathcal{B}$ a $C^*$-subalgebra of $\mathcal{A}$
   with a common unit. Rieffel pointed out in 
   \cite[Theorem]{R1} that the factor norm  $\inf_{b \in \mathcal{B}} \|a-b\|$ obeys the strong Leibniz property, 
   since it equals to a commutator norm. 
   To get connection with the standard deviation, notice that K. Audenaert provided sharp estimate for different types of non-commutative
   (or quantum) deviations determined by matrices \cite{A}. 
   Not long ago Rieffel extended these results to $C^*$-algebras with a completely different approach \cite{R2}. His theorem reads as follows: for any $a \in \mathcal{A},$ 
    $$ \max_{\omega \in \mathcal{S}(\mathcal{A})} \omega(|a-\omega(a)|^2)^{1/2} = \min_{\lambda \in \mathbb{C}} \|a - \lambda {\bf 1}_\mathcal{A}\|,$$ where 
    $\mathcal{S}(\mathcal{A})$ denotes the state space of $\mathcal{A};$ i.e. the set of positive linear functionals of $\mathcal{A}$ with norm $1.$
    For a short proof of this theorem, exploiting the Birkhoff--James orthogonality in operator algebras, the reader might see \cite{BG}.
    The factor norm on the left-hand side above indicates that 'the largest standard deviation' is a strongly Leibniz seminorm. Surprisingly, the standard deviation itself
    is a strongly Leibniz seminorm. Precisely, whenever $\sigma_2^\omega(a) = \omega(|a-\omega(a)|^2)^{1/2},$ the seminorm $\sigma_2^\omega$ on $\mathcal{A}$ is strongly Leibniz if $\omega$ is 
    tracial \cite[Proposition 3.4]{R2}. Moreover, if one defines the non-commutative standard deviation by the formula
    $$ \tilde{\sigma}_2^\omega(a) = \omega(|a-\omega(a)|^2)^{1/2} \vee  \omega(|a^*-\omega(a^*)|^2)^{1/2} ,$$
    then $\tilde{\sigma}_2^\omega$ is strongly Leibniz for any $\omega \in \mathcal{S}(\mathcal{A}),$ see \cite[Theorem 3.5]{R2} (without assuming that $\omega$ is tracial). 
   Quite recently, the equality
    $$\max_{\omega \in \mathcal{S}(\mathcal{A})} \omega(|a-\omega(a)|^k)^{1/k} = 2B_k^{1/k}\min_{\lambda \in \mathbb{C}} \|a - \lambda {\bf 1}_\mathcal{A}\|$$
 was proved in \cite{L} for the $k$th central moments of normal elements, 
 where $k$ is even and $B_k$ denotes the largest $k$th centered moment of the Bernoulli distribution. 
 From this result it follows that 'the largest $k$th moments' in commutative $C^*$-algebras are strongly Leibniz as well.
 
  The aim of the paper is to study whether general or higher-ordered centered moments possess the (strong) Leibniz property in ordinary
  probability spaces, or not. In the next section we shall give
  a rough estimate of the centered moments of products of bounded random variables which gives back Rieffel's statement 
  on the standard deviation. 
  After that we shall present
  some scattered Leibniz-type result for different moments on different (discrete, general) probability spaces. We leave open the
  question whether all centered 
  moments in general probability spaces define a strongly Leibniz seminorm. 
  Lastly, in Section 3, we shall answer affirmatively Rieffel's question on the standard deviation in non-commutative probability spaces.

\section{Leibniz seminorms in function spaces}

  In this section we shall study the Leibniz property and similar estimates in ordinary probability spaces. Let $(\Omega, \mathcal{F}, \mu)$ 
  be a probability space. For any $f \colon \Omega \rightarrow \mathbb{C} \in L^\infty(\Omega,  \mu)$ and $1 \leq p < \infty,$ let us define 
  $$ \sigma_p(f; \mu) = \left(\int_\Omega \left|f - \int_\Omega f \: d\mu\right|^p \: d\mu\right)^{1/p}$$
  and
   $$ \sigma_\infty(f; \mu) = \mbox{ess sup } \left|f - \int_\Omega f \: d\mu\right|. $$
  If no confusion can arise, we simply use the notation $\sigma_p(f).$ Relying on \cite{R2}, we know that the standard deviation
  is a strongly Leibniz seminorm; that is, the inequalities 
    $$ \sigma_2(fg) \leq  \|g\|_\infty \sigma_2(f) + \|f\|_\infty \sigma_2(g) $$  for  $f,g \in L^\infty(\Omega,  \mu),$
  and
   $$\sigma_2(1/f) \leq  \|1/f\|_\infty^2 \sigma_2(f) $$
   whenever $1/f \in L^\infty(\Omega,  \mu)$  hold. For the non-commutative analogues of the result, see \cite{R2}.
  
  We begin with an observation which shows
  that one can reduce the problem of the strongly Leibniz property to that of the discrete uniform distributions.  
   
   \begin{prop}\label{ekvi}
     Fix $1 \leq p < \infty.$ The following statements are equivalent:
     \begin{itemize}
      \item[(i)]  For any probability space $(\Omega, \mathcal{F}, \mu)$, $\sigma_p$ is a strongly Leibniz seminorm on $L^\infty(\Omega, \mu).$
      \item[(ii)] For every $n \in \mathbb{Z}_+,$ $\sigma_p$ is a strongly Leibniz seminorm on $\ell^\infty_n$ endowed with the uniform distribution. 
     \end{itemize}
   \end{prop}
    \begin{proof}
       Obviously, (i) implies (ii).  To see the reverse implication, choose pairwise disjoint sets $S_k \in  \mathcal{F}$ $(1\leq k \leq n).$  As usual
       $\chi_{S_k}$  denotes the characteristic function of the set $S_k.$  Let us consider the measurable simple functions $f_n = \sum_{k=1}^n a_k\chi_{S_k}$ 
       and $g_n = \sum_{k=1}^n b_k\chi_{S_k}$ on $\Omega.$ Let us assume 
       that $\bigcup_{k=1}^n S_k = \Omega,$  so that the constants $\mu(S_k)$ define a probability measure $\mu_n$ on the set $\mathbb{Z}_n = \{1, \hdots, n\}.$ Then for any $\varepsilon > 0$
       we can readily find a probability measure $\nu_n =  (p_1, \hdots, p_n)$ such that $p_i \in \mathbb{Q}$ $(1 \leq i \leq n)$ and the inequalities         
        \begin{align*}
         |\sigma_p(f_n; \mu_n) &- \sigma_p(f_n; \nu_n)| \leq \varepsilon \\    
          |\sigma_p(g_n; \mu_n) &- \sigma_p(g_n; \nu_n)| \leq \varepsilon \\
            |\sigma_p(f_ng_n; \mu_n) &- \sigma_p(f_n g_n; \nu_n)| \leq \varepsilon 
        \end{align*} 
       hold. Now let us choose the integers $m$ and $r_i$ such that $p_i = r_i / m$ for every $1 \leq i \leq n.$ Then the map
         $$\Phi \colon (c_1, \hdots, c_n) \mapsto (\underbrace{c_1, \hdots, c_1}_{r_1}, \hdots, \underbrace{c_n, \hdots, c_n}_{r_n}) $$ 
         defines an isometric algebra homomorphism from $\ell^\infty_n$ into $\ell^\infty_m.$
       Let $\lambda_m$ denote the uniform distribution on the set $\mathbb{Z}_m.$  We clearly have, for instance, $\sigma_p(f_n; \nu_n) = \sigma_p(\Phi(f_n); \lambda_m),$ hence
       $$ \sigma_p(f_n g_n; \nu_n) \leq \|f_n\|_\infty  \sigma_p(g_n; \nu_n) + \|g_n\|_\infty  \sigma_p(f_n; \nu_n)$$ follows as well. Since
       $\varepsilon$ can be arbitrary small, we obtain that $\sigma_p$ is a Leibniz seminorm on $\ell^\infty_n(\mu_n).$
       Now if we choose sequences $\{f_n\}_{n=1}^\infty$ and $\{g_n\}_{n=1}^\infty$
       of measurable simple functions such that $f_n \rightarrow f$ and $g_n 
       \rightarrow g$ in $L^p$ norm, furthermore, $\|f_n\|_\infty = \|f\|_\infty$ and $\|g_n\|_\infty = \|g\|_\infty$ hold for every $n,$
       we infer that $\sigma_p$ has the Leibniz property. 
        A very similar reasoning on the invertible elements gives that $\sigma_p$ is actually strongly Leibniz on $L^\infty(\Omega, \mu).$
    \end{proof}

   Despite of the above equivalence, in arbitrary measure spaces we do not know whether $\sigma_p$ is strongly Leibniz or not. But later we will prove this
   property for $\sigma_\infty$ in the real Banach space $L^\infty(\Omega, \mu; \mathbb{R})$  (see Theorem \ref{main} below). Actually, the second part of
   the section deals with only real-valued functions. In the 
   general situation, we have only a rough Leibniz-type estimate as we shall see below. 
   
   In any $L^p(\Omega, \mu)$ $(1 \leq p \leq \infty)$ space, the projection $P$ is given by the map
     $$ f \mapsto \mathbb{E}f = \int_\Omega f \: d\mu.$$
    
   Then we are able to prove a slight generalization of Rieffel's statement \cite[Proposition 3.4]{R2} in probability spaces.
 
 \begin{prop}\label{bad}
  For any $1 \leq p\leq \infty$ and $f,g \in L^\infty(\Omega,  \mu),$
  we have that 
 $$ \frac{2}{\|I-P\|_p + 1}  \|fg - \mathbb{E}(fg)\|_p \leq  \|g\|_\infty \|f-\mathbb{E}f\|_p + \|f\|_\infty \|g-\mathbb{E}g\|_p.$$
   \end{prop}
 
\begin{proof} First, note that $\|I-P\|_p \geq 1$ (except for the trivial case $I=P$). Hence, without loss of generality, we can assume that
  $$ \|fg - \mathbb{E}(fg)\|_p \geq \max( \|f\|_\infty \|g-\mathbb{E}g\|_p, \|g\|_\infty \|f-\mathbb{E}f\|_p),$$
  otherwise the proof is done. Obviously, 
   $$ \|f(g-\mathbb{E}g) - \mathbb{E}(f(g-\mathbb{E}g))\|_p \leq \|I - P\|_p \|f(g-\mathbb{E}g)\|_p.$$
  From the reversed triangle inequality we obtain that 
   $$ \|fg-\mathbb{E}(fg)\|_p - \|\mathbb{E}f\mathbb{E}g- f\mathbb{E}g\|_p   \leq \|f(g-\mathbb{E}g) - \mathbb{E}(f(g-\mathbb{E}g))\|_p, $$
   which implies that 
   $$ \|fg - \mathbb{E}(fg)\|_p \leq \|I - P\|_p \|f\|_\infty \|g-\mathbb{E}g\|_p + \|g\|_\infty \|f-\mathbb{E}f\|_p.$$
  Changing the variables $f,$ $g$ and summing up the inequalities, we get the statement of the proposition.
 \end{proof}
 
 \begin{remark}
 One can find a non-trivial upper estimate of the constant $\|I-P\|_p.$ For instance,
 if $\Omega= \{ 1, \hdots, n \}$ and $\mu$ is the 
 uniform distribution on $\Omega$, from the definition of the matrix $p$-norms
 one can easily see that $\|I-P\|_1=\|I-P\|_{\infty}=2-\frac2n$ and $\|I-P\|_2=1$. 
 As another example, let us consider the Banach spaces $L^p[0,1]$ endowed
 with the Lebesgue measure. Then a simple calculation shows that $\|I-P\|_1 = \|I-P\|_\infty = 2.$ Moreover,
 $I-P$ is clearly an orthogonal
 projection in $L^2[0,1];$ that is, $\|I-P\|_2 = 1.$ Now a straightforward application of the Riesz--Thorin
 interpolation theorem gives that (see \cite{R})
   $$\|I-P\|_p \leq 2^{|1 - \frac{1}{2p}|}.$$ The projection $I-P$ is actually the minimal projection to the hyperlane 
   $X_p = \{ f \in  L^p[0,1]  : \mathbb{E}f = 0\};$
   i.e. it has the minimal norm among the projections of range $X_p.$ 
  C. Franchetti showed in his paper \cite{F} that 
  $$ \|I-P\|_p = \max_{0\leq x \leq 1} (x^{p-1} + (1-x)^{p-1})^{1/p}(x^{q-1} + (1-x)^{q-1})^{1/q},$$ where $1/p+1/q =1.$
 \end{remark}
  
  \begin{remark} One can apply a derivation approach mentioned in the Introduction to obtain Leibniz-type estimates of the moments
  of invertible functions. To do this, let us renorm the space $L^p(\Omega, \mu),$ $2 \leq p < \infty,$ so that
   $$ \|x\|_{p, \vee} :=  |\mathbb{E}x| + \|x - \mathbb{E}x\|_p.$$ Let $X$ denote the renormed space.
  Define the multiplication operator $M_f \colon x \mapsto fx$ and the derivation $\delta(M_f) = [P, M_f] =  PM_f - M_fP.$
  A straightforward calculation
  yields that 
   $$\| M_f x \|_{p, \vee} \leq \|f\|_\infty \|x\|_p  + \|I-P\|_p \|f\|_\infty \|x\|_p \leq (1+ \|I-P\|_p) \|f\|_\infty \|x\|_{p, \vee}; $$
   that is, $\| M_f\| \leq (1+\|I-P\|_p) \|f\|_\infty.$ Moreover, 
   $\delta(M_f) \mathbb{E}x = \mathbb{E}x(\mathbb{E}f - f) \in (I-P)X,$ thus $\|\delta(M_f)_{|PX}\| = \sigma_p(f).$ 
   On the other hand,  $\delta(M_f) (x - \mathbb{E}x)= \mathbb{E}(fx) - \mathbb{E}f \mathbb{E}x = \mathbb{E}((f-\mathbb{E}f)(x - \mathbb{E}x)).$
   From Hölder's inequality we get that 
   $$ \|\delta(M_f) (x - \mathbb{E}x)\|_{p, \vee} \leq \|f-\mathbb{E}f\|_q \|x-\mathbb{E}x\|_p \qquad (1/q + 1/p = 1)$$
   hence $\|\delta(M_f)_{|(I-P)X}\| \leq \sigma_p(f)$ follows. Since the operator $\delta(M_f)$ interchanges the subspaces $PX$ and $(I-P)X,$ we have
   $$ \|\delta(M_f)\| = \sigma_p(f).$$
   An application of the derivation rules tells us that 
    $$ \sigma_p(1/f) \leq (1+\|I-P\|_p)^2\|1/f\|_\infty^2 \sigma_p(f) $$ holds whenever $1/f \in L^\infty(\Omega, \mu).$
    
   For any $1 \leq p \leq \infty,$ we can get a different estimate from the equality
     $$ (I-P)M_{1/f}(I-P)f = (1/f - \mathbb{E}(1/f))\mathbb{E}f.$$
    Hence we conclude that
     $$ {|\mathbb{E}f|} \sigma_p(1/f) \leq \|I-P\|_p\|1/f\|_\infty \sigma_p(f).$$
   \end{remark}
     
   Much of the rest of the section is devoted to a study of the optimality of the above proposition.
   We begin with the following observation.
 
   \begin{prop}\label{help}
   Let $(\Omega, \mathcal{F}, \mu)$ be a probability space. For any real-valued $f$ and $x \in L^\infty(\Omega,\mu),$ the inequality   
    \begin{equation*}
		\|f\mathbb{E}x - \mathbb{E}(fx)\|_\infty \leq \|x\|_\infty \|f-\mathbb{E}f\|_\infty
		\end{equation*}
		holds. 
  \end{prop}
  
  \begin{proof}
    Without loss of generality, we can assume that $\mathbb{E}f = 0$ holds and $\|f\|_\infty = 1$.
    Note that the function $f \mapsto f\mathbb{E}x - \mathbb{E}(fx)$ is convex on the weak-$*$ compact, convex  set 
     $$ L^\infty_0(\Omega) := \{ f \in L^\infty(\Omega,  \mu) \colon \|f\|_\infty \leq 1 \mbox{ and } \mathbb{E}f = 0\} \subseteq (L^1(\Omega,\mu))^*. $$ 
  Hence, from the Krein--Milman theorem, it is enough to prove the statement if $f$ is an extreme point of $L^\infty_0(\Omega).$
  We claim that the extreme points of $L^\infty_0(\Omega)$ are the functions with essential range $\{-1,1,c \}$ for some $-1< c<1,$
   ($\mu(\{f=c\}) = 0$ might be possible) and 
   \begin{equation}\label{eq1}
     \mathbb{E}f = \mu(\{f=1\})-\mu(\{f=-1\})+c\mu(\{f=c\})=0.    
   \end{equation} 
   Let us choose a measurable subset $A$ of $\Omega$
   such that $\|f\chi_A\|_\infty \leq 1-\varepsilon < 1.$ If $\mu$ is non-atomic ($A$ is not a singleton), we can find a function $g \in L^\infty_0(\Omega)$ 
   satisfying $\|g\|_\infty \leq \varepsilon$ and $g = 0$ a.e. on $\Omega \setminus A$. Since
    $$ f = \frac12 (f+g)+ \frac12 (f-g),$$
    $f$ is an extreme point if and only if $\mu(A) = 0$. When $\mu$ is atomic, the set $A$ might be a singleton, hence our claim follows. 
   
  Now let $f$ be an extreme point of $L^\infty_0(\Omega).$
   Obviously, $\|f-\mathbb{E}f\|_\infty = 1.$  Furthermore, we have
   \begin{align*}
      \|f\mathbb{E}x - \mathbb{E}(fx)\|_\infty &= \max(|\mathbb{E}x - \mathbb{E}(fx)|,|\mathbb{E}x + \mathbb{E}(fx)|,|c\mathbb{E}x-\mathbb{E}(fx)|) \\
                      &=  \max(|\mathbb{E}(x(1 - f))|,|\mathbb{E}(x(1 + f))|,|\mathbb{E}(x(c-f))|)\\
                        &\leq \|x\|_\infty \max(\|1 - f\|_1,\|1 + f\|_1,\|c-f\|_1).
     \end{align*}
   It remains to show that $\max(\|1 - f\|_1,\|1 + f\|_1,\|c-f\|_1)=1$. Clearly, from \eqref{eq1}
	\begin{align*}
	\|1 - f\|_1&=2\mu(\{f=-1\})+|1-c|\mu(\{f=c\})\\&=1-\mu(\{f=1\})+\mu(\{f=-1\})-c\mu(\{f=c\})=1.
	\end{align*}
	Similarly,
	\begin{align*}
	\|1 + f\|_1&=2\mu(\{f=1\})+|1+c|\mu(\{f=c\})\\&=1+\mu(\{f=1\})-\mu(\{f=-1\})+c\mu(\{f=c\})=1,
	\end{align*}
	and lastly we infer that 
   \begin{align*}
	\|c - f\|_1&=|c-1|\mu(\{f=1\})+|c+1|\mu(\{f=-1\})\\&=\mu(\{f=1\})+\mu(\{f=-1\})+c^2\mu(\{f=c\})\leq 1.
	\end{align*}
    The proof is complete.
  \end{proof}

      For the real Banach space $L^\infty(\Omega,\mu; \mathbb{R}),$ we can simply prove that the seminorm $\sigma_\infty$ is strongly Leibniz 
      as we have seen before for the standard deviation.
   
	   \begin{thm}\label{main}
     Let $(\Omega, \mathcal{F}, \mu)$ be a probability space. For the real Banach space $L^\infty(\Omega,\mu; \mathbb{R}),$
      $$ \sigma_\infty(f)= \|f - \mathbb{E}f\|_\infty$$ is a strongly Leibniz seminorm. 
            \end{thm}
  \begin{proof}
    From Proposition \ref{help}, it follows that 
\begin{align*}
\|fg-\mathbb{E}(fg)\|_p&=\|f(g-\mathbb{E}g)+(f\mathbb{E}g-\mathbb{E}(fg))\|_p\\&\leq\|f\|_{\infty}\|g-\mathbb{E}g\|_p+\|g\|_{\infty}\|f-\mathbb{E}f\|_p,
\end{align*}
and 
\[\left\|\frac1f-\mathbb{E}\frac1f\right\|_p=\left\|\frac1f\left(\mathbb{E}\left(f\cdot\frac1f\right)-f\mathbb{E}\frac1f\right)\right\|_p
\leq\left\|\frac1f\right\|_\infty\cdot\left\|\frac1f\right\|_\infty\cdot \|f-\mathbb{E}f\|_p,\]
which is what we intended to have. 
\end{proof}

 Regarding the case of the uniform distributions seen above in Proposition \ref{ekvi}, we are able to prove the analogue of Proposition \ref{help}
 in very particular cases. Let $\lambda_n$ stand for the uniform distribution on $\mathbb{Z}_n.$

 \begin{prop}\label{help2}
   Fix $1 \leq n \leq 4.$ For $1 \leq p < \infty,$ and any real-valued $f, x \in \ell^\infty_n(\lambda_n),$ we have   
    \begin{equation*}\label{seged} 
		\|f\mathbb{E}x - \mathbb{E}(fx)\|_p \leq \|x\|_\infty \|f-\mathbb{E}f\|_p.
		\end{equation*}
  \end{prop}
  
\begin{proof}	
	 First note that the case $\Omega=\mathbb{Z}_1$ is trivial.
	 On the other hand, in case of $\Omega=\mathbb{Z}_2$, one can have arbitrary distribution.
	 Indeed, let $\mu(1)=p_1$ and $\mu(2)=p_2=1-p_1$. Then by simple calculation we obtain
\[f-\mathbb{E}f=(f_1-f_2)\cdot(p_2,-p_1)\]
and
\[f\mathbb{E}x-\mathbb{E}(fx)=(f_1-f_2)\cdot(p_2x_2,-p_1x_1)\]
so the desired inequality follows immediately.

	To prove the remaining cases $\Omega=\mathbb{Z}_3$ and $\Omega=\mathbb{Z}_4,$ let us
	rescale the inequality and assume that $\|x\|_\infty=1.$ Notice that the function
    $$x \mapsto \|f\mathbb{E}x - \mathbb{E}(fx)\|_p$$ is convex on the closed unit ball 
    $\{x\in L^\infty(\Omega,\mu):\|x\|_\infty \leq 1\}$, therefore it suffices to check the inequality only for its extreme points.
    
     First, we turn to the case $\Omega=\mathbb{Z}_3$. Clearly, for $x=(1,1,1)$ even equality holds,
     so after possible rearrangement and multiplication by constants we may assume that $x=(1,1,-1)$. Then 
\[f-\mathbb{E}f=\frac13(2f_1-f_2-f_3,-f_1+2f_2-f_3,-f_1-f_2+2f_3)\]
and
\[f\mathbb{E}x-\mathbb{E}(fx)=\frac13(f_3-f_1,f_3-f_2,2f_3-f_1-f_2).\]
By using the notation $a_1=2f_1-f_2-f_3$ and $a_2=2f_2-f_1-f_3,$ the inequality reduces to the form
\[\left|\frac{2a_1+a_2}3\right|^p+\left|\frac{a_1+2a_2}3\right|^p\leq |a_1|^p+|a_2|^p,\]
which is obviously true from the convexity of the function $t\mapsto |t|^p.$

 Next, let $\Omega=\mathbb{Z}_4$. By symmetry arguments we can assume that $x=(1,1,1,-1)$ or $x=(1,1,-1,-1)$. Set $x=(1,1,1,-1)$. A simple calculation implies that
\[f-\mathbb{E}f=\frac14(a_1,a_2,a_3,a_4),\]
where
\[a_j=3f_j-\sum_{i\neq j}f_i.\]
Moreover,
\[f\mathbb{E}x-\mathbb{E}(fx)=\frac14\begin{pmatrix} +f_1-f_2-f_3+f_4\\-f_1+f_2-f_3+f_4\\-f_1-f_2+f_3+f_4\\-f_1-f_2-f_3+3f_4\end{pmatrix}
=\frac18\begin{pmatrix}-a_2-a_3\\-a_1-a_3 \\-a_1-a_2 \\ 2a_4\end{pmatrix}.\]
Therefore, it is enough to check that
\[\left|\frac{a_2+a_3}2\right|^p+\left|\frac{a_1+a_3}2\right|^p+\left|\frac{a_1+a_2}2\right|^p\leq|a_1|^p+|a_2|^p+|a_3|^p,\]
which follows again by the convexity of the function $t\mapsto|t|^p.$

Lastly, consider the remaining case $x=(1,1,-1,-1)$. Then
\[f\mathbb{E}x-\mathbb{E}(fx)=
\frac14\begin{pmatrix}-f_1-f_2+f_3+f_4\\-f_1-f_2+f_3+f_4\\-f_1-f_2+f_3+f_4\\-f_1-f_2+f_3+f_4\end{pmatrix}=
\frac{1}{16}\begin{pmatrix} -a_1-a_2+a_3+a_4 \\ -a_1-a_2+a_3+a_4  \\ -a_1-a_2+a_3+a_4 \\ -a_1-a_2+a_3+a_4 \end{pmatrix}.\]
Since 
\[4\left|\frac{-a_1-a_2+a_3+a_4}4\right|^p\leq|a_1|^p+|a_2|^p+|a_3|^p+|a_4|^p,\]
by a convexity argument as seen before, we get the statement of the proposition.
\end{proof}
  
\begin{ex} 
The statement of Proposition \ref{help2} does not hold in general. Let $n\geq5$ and $p = 1, $ for instance. Let $x=(1,\dots,1,-1)$ 
 and $f=(1,0,\dots,0,-1)$ in $\ell^\infty_n(\lambda_n).$
 Obviously,  $\mathbb{E}f=0$, $\mathbb{E}x=1-\frac{2}{n}$, $\mathbb{E}(fx)=\frac2n$, $\|x\|_{\infty}=1$, furthermore,
\[\|f-\mathbb{E}f\|_1=\frac{2}n,\]
and
\[\|f\mathbb{E}x-\mathbb{E}(fx)\|_1=\left\|\left(1-\frac4n,-\frac2n,\dots,-\frac2n,-1\right)\right\|_1=\frac{4n-8}{n^2} \qquad (n\geq 5).\]
Thus
\[ \|f\mathbb{E}x-\mathbb{E}(fx)\|_1 = \left(2-\frac4{n}\right) {\|f-\mathbb{E}f\|_1} > {\|f-\mathbb{E}f\|_1}.\]
\end{ex} 

\begin{ex} 
In the case of non-uniform distributions, the inequality of Proposition \ref{help2} is not true even on $\Omega=\{1,2,3\}$.
To see this, define the measure  $\mu(1)=\frac18$, $\mu(2)=\frac34$, $\mu(3)=\frac18$ and consider $f=(1,0,-1)$ and $x=(1,1,-1)$.
Then $\mathbb{E}f=0$, $\mathbb{E}x=\frac34,$ $\mathbb{E}(fx)=\frac14,$ and
\[\|f-\mathbb{E}f\|_1=\frac14,\]
while
\[\|f\mathbb{E}x-\mathbb{E}(fx)\|_1=\left\|\left(\frac12,-\frac14,-1\right)\right\|_1=\frac38.\]
\end{ex}	

As we have seen before in the proof of Theorem \ref{main}, we can infer the next statement on discrete measure spaces. 

\begin{cor}\label{kov}
 For $1 \leq n \leq 4$ and $1 \leq p < \infty,$ the seminorm $\sigma_p$ is strongly Leibniz 
 on the real $\ell_n^\infty$ endowed with uniform
 distribution.
\end{cor}

 Surprisingly, we cannot prove or disprove the last statement on measure spaces which contain more than $4$ atoms. Computer simulations suggest
us that Corollary \ref{kov} might be true for any $n$ which would imply that $\sigma_p$ is a 
strongly Leibniz seminorm for every $1 \leq p < \infty$ (see Proposition \ref{ekvi}). Now we have only a very few particular results on general measure spaces.  
Denote $\lambda_n$ the uniform distribution on the set $\mathbb{Z}_n,$ as usual.

   	\begin{prop} Let $1\leq p < \infty$ and  $f,$  $g\in \ell^\infty_n(\lambda_n)$ be such that the coordinates 
   	of $f,g$ and $fg$ have the same order. Then 
\[\|fg-\mathbb{E}(fg)\|_p\leq \|g\|_{\infty}\|f-\mathbb{E}f\|_p + \|f\|_{\infty}\|g-\mathbb{E}g\|_p\]
holds.
\end{prop}

\begin{proof}

We use the fact that the $\ell^p$ norm with uniform distribution and $1\leq p\leq\infty$ is a Schur-convex function \cite[Ch.~3 Example I.1]{MOA}.
Therefore, it suffices to prove that the vector $fg-\mathbb{E}(fg)$ is majorized by $\|f\|_\infty(g-\mathbb{E}g)+\|g\|_\infty(f-\mathbb{E}f)$. 
To see this, we may assume without loss of generality that $f_1\geq f_2\geq\dots\geq f_n,$ thus we also have $g_1\geq g_2\geq\dots\geq g_n$
and $f_1g_1\geq f_2g_2\geq\dots\geq f_ng_n$. Then we have to verify that
\[\sum_{j=1}^k (f_jg_j-\mathbb{E}(fg))\leq \sum_{j=1}^k\left(\|g\|_\infty(f_j-\mathbb{E}f)+\|f\|_\infty(g_j-\mathbb{E}g)\right),\]
for all $1\leq k\leq n-1$, and equality holds when $k=n$. The latter equality is obvious because both sides are zero if $k=n.$
In the remainder of the proof, a simple calculation gives that
\[n\left(\sum_{j=1}^k(f_j-\mathbb{E}f)\right)
=(n-k)\sum_{j=1}^kf_j-k\sum_{j=k+1}^nf_j=\sum_{j=1}^k\sum_{i=k+1}^n(f_j-f_i)\]
and analogously
\begin{align*}
n\left(\sum_{j=1}^k(f_jg_j-\mathbb{E}(fg))\right)&=\sum_{j=1}^k\sum_{i=k+1}^n(f_jg_j-f_ig_i)\\&=\sum_{j=1}^k\sum_{i=k+1}^n\left(f_j(g_j-g_i)+g_i(f_j-f_i)\right).
\end{align*}
Therefore, it follows that
\begin{align*}
&\sum_{j=1}^k\left(\|f\|_\infty(g_j-\mathbb{E}g)+\|g\|_\infty(f_j-\mathbb{E}f)-(f_jg_j-\mathbb{E}(fg))\right)\\&=\frac1n\left(\sum_{j=1}^k\sum_{i=k+1}^n(g_j-g_i)(\|f\|_\infty-f_j)+(f_j-f_i)(\|g\|_\infty-g_i)\right)\geq0.
\end{align*}
\end{proof}

Analogously to the proof of Proposition \ref{ekvi}, we readily obtain the following corollaries. 

\begin{cor}
   Let $(\Omega, \mathcal{F}, \mu)$ be a probability space and $1\leq p < \infty.$ For any non-negative $f \in L^\infty(\Omega,  \mu),$
 $$ \|f^2-\mathbb{E}f^2\|_p\leq2\|f\|_{\infty}\|f-\mathbb{E}f\|_p. $$
\end{cor}

\begin{cor}
   Let $1\leq p < \infty$ and $\mu$ be a probability measure on the interval $[0,1].$  For any non-negative, bounded and monotone increasing (or decreasing)
   functions $f$ and $g,$ we have
 $$ \|fg-\mathbb{E}fg\|_p\leq \|g\|_{\infty}\|f-\mathbb{E}f\|_p + \|f\|_{\infty}\|g-\mathbb{E}g\|_p. $$
\end{cor}

\section{Standard deviation in $C^*$-algebras}

 In this section we shall complete Rieffel's argument on the standard deviation in non-commutative probability spaces.
 Let $\mathcal{A}$ be a unital $C^*$-algebra and
 denote $\omega$ any faithful state of it. 
 Denote $L^2(\mathcal{A}, \omega)$ the GNS Hilbert space obtained by completing $\mathcal{A}$ for the inner product 
  $\langle a,b \rangle = \omega(b^* a),$ as usual.
 Obviously, every $a \in \mathcal{A}$ has a natural representation; i.e. the left-regular 
 representation $L_a,$ in the operator algebra of $L^2(\mathcal{A}, \omega).$
 Consider now the projection (or Dirac operator) $E \colon a \mapsto \omega(a){\bf 1}_\mathcal{A}$ on $L^2(\mathcal{A}, \omega).$ Direct calculations for 
 the norm of the commutator $\delta(L_a) = [E,L_a] = EL_a - L_aE$ give that
   $$ \|\delta(L_a)\| = \omega(|a - \omega(a)|^2)^{1/2} \vee \omega(|a^* - \omega(a^*)|^2)^{1/2}. $$
 Thus it immediately follows that Rieffel's non-commutative standard deviation is a strongly Leibniz $*$-seminorm, see \cite[Theorem 3.7]{R2}. Moreover, an application
 of the 'independent copies trick' in $C^*$-algebras gives that
  $$ \sigma^\omega_2(a) := \omega(|a - \omega(a)|^2) $$ is strongly Leibniz as well if one assumes that $\omega$ is tracial \cite[Proposition 3.6]{R2}. Actually, the 'strong' part of the statement requires only the tracial assumption.
  Computer simulations for matrices indicate that $\sigma^\omega_2$ might be strongly Leibniz for any state $\omega$ but the question remained
  open in \cite{R2}. Now we shall provide the affirmative answer by means of an elementary argument. 
  
  Pick a faithful state $\omega$ of $\mathcal{A}.$ Let  $\| a\|_2 = \omega(|a|^2)^{1/2}$ denote the norm on $L^2(\mathcal{A}, \omega).$ 
  We begin with 
  
     \begin{lemma}\label{ilem}
      For any $a$ and $x \in \mathcal{A},$ 
    \begin{equation*}
		\|\omega(x)a - \omega(xa)\|_2 \leq \|x\| \|a-\omega(a)\|_2.
		\end{equation*}
  \end{lemma}
  
  \begin{proof} 
   There is no loss of generality in assuming that $\omega(a) =  0.$
   Denote $E$ the orthogonal projection from $L^2(\mathcal{A}, \omega)$ onto its subspace $\mathbb{C} {\bf 1}_\mathcal{A}.$ Then
   $$ \|\omega(x)a - \omega(xa)\|_2 = \|\omega(x)(I-E)a - E\omega(xa)\|_2 = \|\omega(x)a\|_2 + |\omega(xa)|.$$ Notice that the Cauchy--Schwarz
   inequality readily gives that 
   $$ |\omega(xa)| = |\omega((x - \omega(x))a)| \leq \|a\|_2 \|x^* - \omega(x^*)\|_2.$$
   Hence 
   \begin{align*}
    \|\omega(x)a - \omega(xa)\|_2 &= \|\omega(x)a\|_2 + |\omega(xa)| \\
                                        &\leq |\omega(x^*)| \|a\|_2  +  \|a\|_2 \|x^* - \omega(x^*)\|_2 \\
                                        &= \|x^*\|_2\|a\|_2 \\
                                        &\leq \|x^*\| \|a\|_2 \\
                                        &= \|x\| \|a\|_2,
    \end{align*}
    and the proof is finished.
   \end{proof} 
         
      Now the main theorem of the section reads as follows.  
         
  \begin{thm}
      For any invertible $a \in \mathcal{A},$ the inequality   
    \begin{equation*}
		\|a^{-1} - \omega(a^{-1})\|_2 \leq \|a^{-1}\|^2 \|a-\omega(a)\|_2
		\end{equation*}
		holds.

  \end{thm}   
 
 \begin{proof}
   We clearly have that $$\|xa\|_2 \leq \|x\| \|a\|_2$$ for any $x \in  \mathcal{A}.$ In fact,
    $$ \omega(|xa|^2) = \omega(a^*|x|^2a) \leq \omega(a^*\|x\|^2a) = \|x\|^2 \omega(|a|^2).  $$
   Combining the previous inequality with Lemma \ref{ilem}, it follows that
    \begin{align*}
  \|a^{-1} - \omega(a^{-1})\|_2 &= \|a^{-1}(\omega(a^{-1}a )- \omega(a^{-1})a))\|_2 \\
                                &= \|a^{-1}(\omega({a}^{-1}a )- \omega({a}^{-1})a))\|_2 \\
                                &\leq  \|a^{-1}\| \|\omega({a}^{-1} a)- \omega({a}^{-1})a)\|_2 \\
                                &\leq \|a^{-1}\|^2 \|a-\omega(a)\|_2,
  \end{align*}
    and the proof is complete.
      \end{proof} 
      
  With \cite[Proposition 3.4]{R2} at hand, we immediately obtain the following
  
  \begin{thm}
    Let $\mathcal{A}$ be a unital $C^*$-algebra. For any faithful state $\omega$ of $\mathcal{A},$ $\sigma^\omega_2(a)$ is a strongly Leibniz seminorm.
  \end{thm}
  
  Alternatively, for any faithful tracial state $\omega,$ we can define a derivation on a Banach algebra to infer the above corollary. 
  In fact, let us consider the Banach space $$\mathcal{A} \oplus L^2(\mathcal{A}, \omega)$$
  endowed with the norm $\|(x,y)\| = \max(\|x\|, \|y\|_2).$ The linear operators
   $$  E \colon (x,y) \mapsto (0, \omega(x){\bf 1}_\mathcal{A})$$ and
   $$ T_a \colon (x,y) \mapsto (xa, ya)$$ on $\mathcal{A} \oplus L^2(\mathcal{A}, \omega)$ define a
   strongly Leibniz seminorm $L$ on $\mathcal{A}$ via the norm of the derivation $L(a) = \|\delta(T_{a})\| = \|[E,T_{a}]\|.$
   From Lemma \ref{ilem}, we have that
    $$\|\delta(T_{a})(x,y)\| = \|(0, \omega(x)a - \omega(xa))\| \leq \|a-\omega(a)\|_2 \|x\| \leq \|a-\omega(a)\|_2 \|(x,y)\|.$$
   With the choice of $({\bf 1}_\mathcal{A},0)$, we get 
             $$\|\delta(T_{a})\| = \|a-\omega(a)\|_2.$$   
  Since $\omega$ is tracial, $\|xa\|_2 \leq \|a\|\|x\|_2.$ Hence it clearly follows that $\|T_{a}\| = \|a\|.$ Notice that $T_{ab} = T_{a} T_{b}.$  Now a direct application
  of the derivation rules gives that $\|\delta(T_{a})\| = L(a) = \sigma^\omega_2(a)$ is a strongly Leibniz seminorm.

\end{document}